\documentclass[11pt,reqno]{amsart}
\usepackage{amsmath, amsfonts, amssymb}
\usepackage[cp1251]{inputenc}
\usepackage[english]{babel}
\usepackage{amscd,amssymb,amsmath,amsthm}
\usepackage{graphicx}
\usepackage{comment}
\usepackage{color}
\usepackage{cite}
\usepackage{geometry}
\newtheorem{thm}{Theorem}
\newtheorem{defn}{Definition}

\newtheorem{pro}{Proposition}

\numberwithin{equation}{section} 

\begin{document}
\title[]{GRADIENT GIBBS MEASURES WITH PERIODIC BOUNDARY LAWS OF A GENERALIZED SOS MODEL ON A CAYLEY TREE: }

\author{F.~H.~Haydarov, R.~A.~Ilyasova}

\address{F.~H.~Haydarov$^{1,2,3,4}$, R.~A.~Ilyasova$^{4}$}

\address{1. New Uzbekistan University,
54, Mustaqillik Ave., Tashkent, 100007, Uzbekistan,}

\address{2. AKFA University, 17, Kichik Halqa Yuli Street, Tashkent city,
100095, Uzbekistan.}

\address{3. Institute of Mathematics,
9, University str., Tashkent, 100174, Uzbekistan,}

\address{4. National University of Uzbekistan,
 University str., 4 Olmazor district, Tashkent, 100174, Uzbekistan.}

  \email{haydarov\_imc@mail.ru}
 \email{ilyasova.risolat@mail.ru}
\begin{abstract}
 We consider Gradient Gibbs measures corresponding to a periodic boundary law for a generalized SOS model with spin values from a countable set, on Cayley trees. On the Cayley tree, detailed information on Gradient Gibbs measures for models of SOS type are given in \cite{3, 16,8,11} and we continue the works for the generalized SOS model. Namely, in this paper the problem of finding Gradient Gibbs measures which correspond to periodic boundary laws is reduced to a functional equation and by solving the equation all Gradient Gibbs measures with 4 periodic boundary laws are found.
\end{abstract}
\maketitle



\section{Preliminaries}

The gradient Gibbs measure is a probability measure on the space of gradient fields defined on a manifold. It is often used in statistical mechanics to describe the equilibrium states of a system. The gradient Gibbs measure is derived from the Gibbs measure, which is a probability measure on the space of field configurations. The critical difference is that the gradient Gibbs measure focuses on the gradients of the fields rather than the fields themselves (e.g. \cite{7}).

Specifically, the Gradient Gibbs measure is defined on the set of spin configurations of a system on a Cayley tree.  The Gradient Gibbs measure on a Cayley tree assigns a probability to each possible spin configuration based on the energy of that configuration. The energy of a spin configuration is determined by the interactions between neighboring spins. In the case of a Cayley tree, each spin is coupled to its nearest neighbors along the edges of the tree (see \cite{5}).

Mathematically, the gradient Gibbs measure assigns a probability to each possible configuration of a gradient field on the Cayley tree, based on an energy function. The energy function typically represents the interactions between the gradients of a scalar field or a vector field. The probability of a configuration is proportional to the exponential of the negative energy of that configuration (e.g. \cite{1,4,5, 12, 14}).

 The study of random field $\xi_x$ from a lattice graph (e.g., $\mathbb{Z}^d$ or a Cayley tree $\Gamma^k$ ) to a measure space $(E, \mathcal{E})$ is a central component of ergodic theory and statistical physics.
In many classical models from physics (e.g., the Ising model, the Potts model, the SOS model), $E$ is a finite set (i.e., with a finite underlying measure $\lambda$ ), and $\xi_x$ has a physical interpretation as the spin of a particle at location $x$ in a crystal lattice (detail in \cite{1, 2, 3, 6, 7, 8, 9, 10, 14, 15}).

Let us give basic definitions and some known facts related to (gradient) Gibbs measures. The Cayley tree $\Gamma^{k}=(V, L)$ of order $k \geq 1$ is an infinite tree, i.e. connected and undirected graph without cycles, each vertex of which has exactly $k+1$ edges. Here $V$ is the set of vertices of $\Gamma^{k}$ va $L$ is the set of its edges.

Consider models where the spin takes values in the set $\Phi\subseteq \mathbb{R}_{\infty}^{+}$,
and is assigned to the vertices of the tree. Let $\Omega_A=\Phi^A$ be the set of all configurations
on $A$ and $\Omega:=\Phi^V$. A partial order $\preceq$ on $\Omega$ defined pointwise by stipulating that $\sigma_1 \preceq \sigma_2$ if and only if $\sigma_1(x) \leq \sigma_2(x)$ for all $x \in V$. Thus $(\Omega, \preceq)$ is a poset, and whenever we consider $\Omega$ as a poset then it will always be with respect to this partial order. The poset $\Omega$ is complete. Also, $\Omega$ can be considered as a metric space with respect to the metric $\rho: \Omega \times \Omega \rightarrow \mathbb{R}^{+}$ given by
$$
\rho\left(\left\{\sigma(x_n)\right\}_{x_n \in V},\left\{\sigma^{\prime}(x_n)\right\}_{x_n \in V}\right)=\sum_{n \geq 0} 2^{-n}\mathcal{X}_{\sigma(x_n)\neq\sigma^{\prime}(x_n)},
$$
where $V=\{x_0, x_1, x_2, ....\}$ and $\mathcal{X}_A$ is the indicator function.

We denote by $\mathcal{N}$ the set of all finite subsets of $V$. For each $A\in V$ let $\pi_A: \Omega \rightarrow\Phi^{A}$ be given by $\pi_A\left(\sigma_x\right)_{x\in V}=\left(\sigma_x\right)_{x\in A}$ and let $\mathcal{C}_A=\pi_A^{-1}\left(\mathcal{P}\left(\Phi^{A}\right)\right)$. Let $\mathcal{C}=\bigcup_{A \in \mathcal{N}} \mathcal{C}_A$ and $\mathcal{F}$ is the smallest sigma field containing $\mathcal{C}$. Write $\mathcal{T}_{\Lambda}=\mathcal{F}_{\mathrm{V} \backslash \Lambda}$ and $\mathcal{T}$ for the tail- $\sigma$-algebra, i.e., intersection of $\mathcal{T}_{\Lambda}$ over all finite subsets $\Lambda$ of $\mathbb{L}$ the sets in $\mathcal{T}$ are called tail-measurable sets.

\begin{defn}\label{Definition 6.14.}\cite{5} Let $P_\Lambda: \Omega \rightarrow \overline{\mathbb{R}}:=\mathbb{R}\cup\{-\infty, \infty\}$ be $\mathcal{F}_\Lambda$-measurable mapping for all $\Lambda \in \mathcal{N}$, then the collection $P=\left\{P_\Lambda\right\}_{\Lambda\in \mathcal{N}}$ is called \textbf{a potential}. Also, the following expression
\begin{equation}
H_{\Delta, P}(\sigma) \stackrel{\text { def }}{=} \sum_{\Delta \cap \Lambda \neq \varnothing, \Lambda \in \mathcal{N}} P_\Lambda(\sigma), \quad \forall \sigma \in \Omega.
\end{equation}
is called \textbf{Hamiltonian} $H$ associated to the potential $P$.
\end{defn}

For a fixed inverse temperature $\beta>0$, the Gibbs specification is determined by a family of probability kernels $\zeta=\left(\zeta_{\Lambda}\right)_{\Lambda \in \mathcal{N}}$ defined on $\Omega_{\Lambda} \times \mathcal{F}_{\Lambda^c}$ by the Boltzmann-Gibbs weights
\begin{equation}\label{xino}
\zeta_{\Lambda}(\sigma_{\Lambda}\mid\omega)=\frac{1}{\mathbf{Z}_{\Lambda}^\omega} e^{-\beta H_{\Lambda, P}^\omega\left(\sigma_{\Lambda}\omega\right)}
\end{equation}
where $\mathbf{Z}_{\Lambda}^\omega=\sum_{\sigma \in \Omega_{\Lambda}} e^{-\beta H_{\Lambda, P}^\omega\left(\sigma_{\Lambda}\right)}$ is the partition function, related to free energy.

 From \cite{5}, the family of mappings $\{\zeta_{\Lambda}(\sigma \mid \omega)\}_{\Lambda\in\mathcal{N}}$ is the family of  proper $\mathcal{F}_\Lambda$-measurable quasi-probability kernels. Thus, the collection $\mathcal{V}=\left\{\zeta_\Lambda\right\}_{\Lambda \in \mathcal{N}}$ will be called an $\mathbb{F}$-\emph{\textit{specification}} if $\zeta_{\Delta}=\zeta_{\Delta} \zeta_{\Lambda}$ whenever $\Lambda, \Delta \in \mathcal{N}$ with $\Lambda \subseteq \Delta$.  Let $\mathcal{V}=\left\{\zeta_\Lambda\right\}_{\Lambda \in \mathcal{N}}$ be an $\mathbb{F}$-specification; then a probability measure $\mu \in \mathrm{P}(\mathcal{F})$ is called a \emph{\textit{Gibbs measure with specification $\mathcal{V}$}} if $\mu=\mu \zeta_\Lambda$ for each $\Lambda \in \mathcal{N}$.

\section{Gradient Gibbs measure}

For any configuration $\omega=(\omega(x))_{x \in V} \in \mathbb{Z}^{V}$ and edge $e=\langle x, y\rangle$ of $\vec{L}$ (oriented) the difference along the edge $e$ is given by $\nabla \omega_e=\omega_y-\omega_x$ and $\nabla \omega$ is called the gradient field of $\omega$.
The gradient spin variables are now defined by $\eta_{\langle x, y\rangle}=\omega_y-\omega_x$ for each $\langle x, y\rangle$.
The space of gradient configurations is denoted by $\Omega^{\nabla}$. The measurable structure on the space $\Omega^{\nabla}$ is given by $\sigma$-algebra
$$
\mathcal{F}^{\nabla}:=\sigma\left(\{\eta_e \mid e \in \vec{L}\}\right).
$$
Note that $\mathcal{F}^{\nabla}$ is the subset of $\mathcal{F}$ containing those sets that are invariant under translations $\omega \rightarrow \omega+c$ for $c \in E$.
Similarly, we define
$$
\mathcal{T}_{\Lambda}^{\nabla}=\mathcal{T}_{\Lambda} \cap \mathcal{F}^{\nabla}, \ \mathcal{F}_{\Lambda}^{\nabla}=\mathcal{F}_{\Lambda} \cap \mathcal{F}^{\nabla}
$$
For nearest-neighboring (n.n.) interaction potential $\Phi=\left(\Phi_b\right)_b$, where $b=\langle x, y\rangle$ is an edge, define symmetric transfer matrices $Q_b$ by
$$
Q_b\left(\omega_b\right)=e^{-\left(\Phi_b\left(\omega_b\right)+|\partial x|^{-1} \Phi_{\{x\}}\left(\omega_x\right)+|\partial y|^{-1} \Phi_{\{w\}}\left(\omega_\nu\right)\right)}.
$$
Define the Markov (Gibbsian) specification as
$$
\gamma_{\Lambda}^{\Phi}\left(\sigma_{\Lambda}=\omega_{\Lambda} \mid \omega\right)=\left(Z_{\Lambda}^{\Phi}\right)(\omega)^{-1} \prod_{b \cap \Lambda \neq 0} Q_b\left(\omega_b\right) .
$$
If for any bond $b=\langle x, y\rangle$ the transfer operator $Q_b\left(\omega_b\right)$ is a function of gradient spin variable $\zeta_b=\omega_y-\omega_x$ then the underlying potential $\Phi$ is called a gradient interaction potential. Note that for all $A \in \mathcal{F}^{\nabla}$, the kernels $\gamma_{\Lambda}^{\Phi}(A, \omega)$ are $\mathcal{F}^{\nabla}$-measurable functions of $\omega$, it follows that the kernel sends a given measure $\mu$ on $\left(\Omega, \mathcal{F}^{\nabla}\right)$ to another measure $\mu \gamma_{\Lambda}^{\Phi}$ on $\left(\Omega, \mathcal{F}^{\nabla}\right)$. A measure $\mu$ on $\left(\Omega, \mathcal{F}^{\nabla}\right)$ is called a \emph{gradient Gibbs measure} if it satisfies the equality $\mu\gamma_{\Lambda}^{\Phi}=\mu$ (detail in \cite{10,11,13}).

Note that, if $\mu$ is a Gibbs measure on $(\Omega, \mathcal{F})$, then its restriction to $\mathcal{F}^{\nabla}$ is a gradient Gibbs measure.

A boundary law is called $q$-periodic if $l_{x y}(i+q)=l_{x y}(i)$ for every oriented edge $\langle x, y\rangle \in$ $\vec{L}$ and each $i \in \mathbb{Z}$.

It is known that there is a one-to-one correspondence between boundary laws and tree indexed Markov chains if the boundary laws are normalisable in the sense of Zachary [15]:

\begin{defn}\label{Definition 3.4} (Normalisable boundary laws). A boundary law $l$ is said to be normalisable if and only if
$$
\sum_{\omega_x \in \mathbb{Z}}\left(\prod_{z \in \partial x} \sum_{\omega_z \in \mathbb{Z}} Q_{z x}\left(\omega_x, \omega_z\right) l_{z x}\left(\omega_z\right)\right)<\infty
$$
for any $x \in V$.\end{defn}
The correspondence now reads the following:

\begin{thm}\label{Theorem 3.5} (Theorem 3.2 in \cite{15}). For any Markov specification $\gamma$ with associated family of transfer matrices $\left(Q_b\right)_{b \in L}$ we have

1. Each normalisable boundary law $\left(l_{x y}\right)_{x, y}$ for $\left(Q_b\right)_{b \in L}$ defines a unique tree-indexed Markov chain $\mu \in \mathcal{G}(\gamma)$ via the equation given for any connected set $\Lambda \in \mathcal{S}$
\begin{equation}\label{1.4}
\mu\left(\sigma_{\Lambda \cup \partial \Lambda}=\omega_{\Lambda \cup \partial \Lambda}\right)=\left(Z_{\Lambda}\right)^{-1} \prod_{y \in \partial \Lambda} l_{y y_{\Lambda}}\left(\omega_y\right) \prod_{b \cap \Lambda \neq \emptyset} Q_b\left(\omega_b\right),
\end{equation}
where for any $y \in \partial \Lambda, y_{\Lambda}$ denotes the unique n.n. of $y$ in $\Lambda$.

2. Conversely, every tree-indexed Markov chain $\mu \in \mathcal{G}(\gamma)$ admits a representation of the form (3.15) in terms of a normalisable boundary law (unique up to a constant positive factor).
\end{thm}

The Markov chain $\mu$ defined in (\ref{1.4}) has the transition probabilities
\begin{equation}\label{1.5}
P_{x y}(i, j)=\mu\left(\sigma_y=j \mid \sigma_x=i\right)=\frac{l_{y x}(j) Q_{y x}(j, i)}{\sum_s l_{y x}(s) Q_{y x}(s, i)} .
\end{equation}
The expressions (\ref{1.5}) may exist even in situations where the underlying boundary law $\left(l_{x y}\right)_{x, y}$ is not normalisable. However, the Markov chain given by (\ref{1.5}), in general, does not have an invariant probability measure. Therefore in $[8],[9],[10],[11]$ some nonnormalisable boundary laws are used to give gradient Gibbs measures.

Now we give some results of above-mentioned papers. Consider a model on Cayley tree $\Gamma^k=(V, \vec{L})$, where the spin takes values in the set of all integer numbers $\mathbb{Z}$. The set of all configurations is $\Omega:=\mathbb{Z}^V$.

Now we consider the following Hamiltonian:
\begin{equation}\label{H} H(\sigma)=-J\sum_{\langle x,y \rangle}\alpha(| \sigma_{x}-\sigma_{y}|)| \sigma_{x}-\sigma_{y}|,\end{equation}
where
\[ \alpha(|m|)=
 \begin{cases}
  p, & \text{if $m\in 2\mathbb{Z}$} \\
  q, & \text{if $m\in 2\mathbb{Z}+1$}.
   \end{cases}\]
Note that if $p=q$ then the considered model is called SOS model.

For the Hamiltonian (\ref{1.5}) the transfer operator is defined by
$$Q(i,j)=e^{-\beta J \alpha(| i-j|)| i-j|},$$
where $\beta>0$ is the inverse temperature and $J\in\mathbb{R}$.

Also, the boundary law equation of the Hamiltonian can be written as:
\begin{equation}\label{sq}z_{i}=\left(\frac{Q(i,0)+\sum_{j\in Z_{0}}Q(i,j)z_{j}}{Q(0,0)+\sum_{j\in Z_{0}}Q(0,j)z_{j}}\right)^{k}.\end{equation}

 Put $\theta:=\exp (-J \beta)<1$. For translation invariant boundary law, the transfer operator $Q$ reads $Q(i-j)=\theta^{|i-j|}$ for any $i, j \in \mathbb{Z}$.
If $\theta :=e^{-J \beta}<1$ then we can write the equation (\ref{sq}) as
\begin{equation}\label{h}
z_{i}=\left(\frac{\theta^{\alpha(|i|)|i|}+\sum_{j\in Z_{0}}\theta^{\alpha(|i-j|)|i-j|} z_{j}}{1+\sum_{j\in Z_{0}}\theta^{\alpha(|j|)|j|} z_{j}}\right)^k, \ i\in \mathbb{Z}_{0}:=\mathbb{Z}\setminus\{0\}.
\end{equation}
Let $\{z_i\}_{i\in\mathbb{Z}}$ be $q$-periodic sequence, i.e. $z_i=z_{i+q}$ for all $i\in\mathbb{Z}$.

\begin{equation}\label{newh}
\left\{
  \begin{array}{ll}
    z_{1}=\left(\frac{\theta^{\alpha(|1|)}+\sum_{j\in Z_{0}}\theta^{\alpha(|1-j|)|1-j|} z_{j}}{1+\sum_{j\in Z_{0}}\theta^{\alpha(|j|)|j|} z_{j}}\right)^k; \\
    z_{2}=\left(\frac{\theta^{2\alpha(|2|)}+\sum_{j\in Z_{0}}\theta^{\alpha(|2-j|)|2-j|} z_{j}}{1+\sum_{j\in Z_{0}}\theta^{\alpha(|j|)|j|} z_{j}}\right)^k; \\
    ... \qquad ... \qquad ... \qquad ... \qquad ... \\
    z_{q}=\left(\frac{\theta^{q\alpha(|q|)}+\sum_{j\in Z_{0}}\theta^{\alpha(|q-j|)|q-j|} z_{j}}{1+\sum_{j\in Z_{0}}\theta^{\alpha(|j|)|j|} z_{j}}\right)^k.
  \end{array}
\right.
\end{equation}
Let $u_{i}=u_{0}\sqrt[k]{z_{i}}$ for some $u_{0}>0$ then since (\ref{h}) we obtain
$$u_{i}=\frac{...+\theta^{3q}u_{i-3}^{k}+\theta^{2p}u_{i-2}^{k}+\theta^{q}u_{i-1}^{k}+u_{i}^{k}+\theta^{q}u_{i+1}^{k}+\theta^{2p}u_{i+2}^{k}+\theta^{3q}u_{i+3}^{k}+...}
{...+\theta^{3q}u_{-3}^{k}+\theta^{2p}u_{-2}^{k}+\theta^{q}u_{-1}^{k}+u_{0}^{k}+\theta^{q}u_{1}^{k}+\theta^{2p}u_{2}^{k}+\theta^{3q}u_{3}^{k}+...}.
$$
We rewrite the last system of equations in the following form:
\begin{equation}\label{hh}
u_{i}=\frac{\sum_{j=1}^{\infty} \theta^{2pj}u_{i-2j}^{k}+\sum_{j=1}^{\infty} \theta^{(2j-1)q}u_{i-2j+1}^{k}+u_{i}^{k}+\sum_{j=1}^{\infty} \theta^{(2j-1)q}u_{i+2j-1}^{k}+\sum_{j=1}^{\infty} \theta^{2pj}u_{i+2j}^{k}}{\sum_{j=1}^{\infty} \theta^{2pj}u_{-2j}^{k}+\sum_{j=1}^{\infty} \theta^{(2j-1)q}u_{-2j+1}^{k}+u_{0}^{k}+\sum_{j=1}^{\infty} \theta^{(2j-1)q}u_{2j-1}^{k}+\sum_{j=1}^{\infty} \theta^{2pj}u_{2j}^{k}},
\end{equation}
where $i\in \mathbb{Z}.$

\begin{pro}\label{new} Let $\{z_i\}_{i\in\mathbb{Z}}$ be $q$-periodic sequence. Then finding $q$-periodic solutions to the system (\ref{h}) is equivalent to  solving the system of equations (\ref{newh}).
\end{pro}
\begin{proof} To prove the Proposition, it is sufficient to show $z_{i}=z_{q+i}$ for all $i\in\{1,2,...,q-1,q\}$. Since $z_{0}=0$, for a fixed $i_0\in \mathbb{Z}$, the numerator of the fraction in (\ref{h}) can be written as
$$\theta^{\alpha(|i_0|)|i_0|}+\sum_{j\in \mathbb{Z}_{0}}\theta^{\alpha(|i_0-j|)|i_0-j|} z_{j}=\sum_{j\in \mathbb{Z}}\theta^{\alpha(|i_0-j|)|i_0-j|} z_{j}.$$
Also, it can be rewritten as
\begin{equation}\label{s}
\sum_{j\in \mathbb{Z}}\theta^{\alpha(|i_0-j|)|i_0-j|} z_{j}=...+\theta^{2p}z_{i_0-2}+\theta^{q}z_{i_0-1}+z_{i_0}+\theta^{q}z_{i_0+1}+\theta^{2p}z_{i_0+2}+...
\end{equation}
Similarly, for $i_0+q$ we have
\begin{equation}\label{ss}
\sum_{j\in \mathbb{Z}}\theta^{\alpha(|i_0+q-j|)|i_0+q-j|} z_{j}=...+\theta^{2p}z_{i_0+q-2}+\theta^{q}z_{i_0+q-1}+z_{i_0+q}+\theta^{q}z_{i_0+q+1}+\theta^{2p}z_{i_0+q+2}+...
\end{equation}
If we change $z_{k+q}$ in (\ref{ss}) to $z_k$ for all $k\in\mathbb{N}$ then we obtain (\ref{s}). Namely, we have proved
$$z_{i_{0}}=\left(\frac{\theta^{i_{0}\alpha(|i_{0}|)}+\sum_{j\in Z_{0}}\theta^{\alpha(|i_{0}-j|)|i_{0}-j|} z_{j}}{i_{0}+\sum_{j\in Z_{0}}\theta^{\alpha(|j|)|j|} z_{j}}\right)^k=\qquad \qquad  \qquad \ \ \ \ \ \ \ \ \ \  \ \ \ \ \  \ \ \  \ \ \ \ \ \ \ \ \ $$
$$\ \ \  \ \ \ \ \ \ \ \ \ \ \ \ \ \ \ \ \ \ \qquad =\left(\frac{(\theta^{i_{0}+q)\alpha(|i_{0}+q|)}+\sum_{j\in Z_{0}}\theta^{\alpha(|i_{0}+q-j|)|i_{0}+q-j|} z_{j}}{i_{0}+q+\sum_{j\in Z_{0}}\theta^{\alpha(|j|)|j|} z_{j}}\right)^k=z_{i_0+q}.
$$\end{proof}

\section{4-periodic boundary laws for $k\geq 2$}
In this section, we find periodic solutions (defined in \cite{16}) to (\ref{h}) which correspond to periodic boundary condition. Namely, for all $m\in \mathbb{Z}$ we consider the following sequence:
\begin{equation}\label{new1}
 u_{n}=\begin{cases}
  1,  \ \emph{if } n=2m; \\
 a,  \ \emph{if } n=4m-1; \\
 b,  \ \emph{if } n=4m+1, \\
   \end{cases}
\end{equation}
where $a$ and $b$ are some positive numbers.

By Proposition \ref{new}, finding solutions that are formed in (\ref{new1}) to (\ref{h}) is equivalent to solving the following system of equations:

$$a=\frac{...+\theta^{4p}a^{k}+\theta^{3q}+\theta^{2p}b^{k}+\theta^{q}+a^{k}+\theta^{q}+\theta^{2p}b^{k}+\theta^{3q}+\theta^{4p}a^{k}+...}
  {...+\theta^{4p}+\theta^{3q}b^{k}+\theta^{2p}+\theta^{q}a^{k}+1+\theta^{q}b^{k}+\theta^{2p}+\theta^{3q}a^{k}+\theta^{4p}+...};$$

$$b=\frac{...+\theta^{4p}b^{k}+\theta^{3q}+\theta^{2p}a^{k}+\theta^{q}+b^{k}+\theta^{q}+\theta^{2p}a^{k}+\theta^{3q}+\theta^{4p}b^{k}+...}
  {...+\theta^{4p}+\theta^{3q}b^{k}+\theta^{2p}+\theta^{q}a^{k}+1+\theta^{q}b^{k}+\theta^{2p}+\theta^{3q}a^{k}+\theta^{4p}+...}.$$
Namely,

$$a=\frac{2(\theta^{q}+\theta^{3q}+...)+(1+2\theta^{4p}+2\theta^{8p}+...)a^{k}+2(\theta^{2p}+\theta^{6p}+...)b^{k}}
{1+2\theta^{2p}+2\theta^{4p}+...+(\theta^{q}+\theta^{3q}+...)(a^{k}+b^{k})};$$

$$b=\frac{2(\theta^{q}+\theta^{3q}+...)+(1+2\theta^{4p}+2\theta^{8p}+...)b^{k}+2(\theta^{2p}+\theta^{6p}+...)a^{k}}
{1+2\theta^{2p}+2\theta^{4p}+...+(\theta^{q}+\theta^{3q}+...)(a^{k}+b^{k})}.$$

Taking into account $\theta <1$ one writes the last system of equations as follows:
\begin{equation} \label{se1}
 a=\frac{\frac{2\theta^{q}}{1-\theta^{2q}}+\frac{1+\theta^{4p}}{1-\theta^{4p}}a^{k}+\frac{2\theta^{2p}}{1-\theta^{4p}}b^{k}}{\frac{1+\theta^{2p}}{1-\theta^{2p}}+\frac{\theta^{q}}{1-\theta^{2q}}(a^{k}+b^{k})}, \quad  b=\frac{\frac{2\theta^{q}}{1-\theta^{2q}}+\frac{1+\theta^{4p}}{1-\theta^{4p}}b^{k}+\frac{2\theta^{2p}}{1-\theta^{4p}}a^{k}}{\frac{1+\theta^{2p}}{1-\theta^{2p}}+\frac{\theta^{q}}{1-\theta^{2q}}(a^{k}+b^{k})}.
 \end{equation}

For all $k\in\mathbb{N}$ and $p,q\in \mathbb{R}$, the analysis of (\ref{se1}) is so difficult and that's why we consider the case $2p=q=1$ and $k=2.$ Then (\ref{se1}) can be written as
\begin{equation}\label{uff}a=\frac{\tau a^{2}+2b^{2}+2}{a^{2}+b^{2}+\tau +2}, \quad b=\frac{\tau b^{2}+2a^{2}+2}{a^{2}+b^{2}+\tau +2},\end{equation}
where $\tau=\theta+\frac{1}{\theta}>2.$

At first, we consider the case $a=b$. The system of equations (\ref{uff}) is reduced to the polynomial equation:
\begin{equation}\label{cc}
2a^{3}-(\tau+2)a^{2}+(\tau+2)a-2=0.
\end{equation}
Since the last equation has a solution $a=1$, we divide both sides of (\ref{cc}) by $a-1$. Consequently, one gets
$$2a^{2}-\tau a+2=0.$$
If $\tau > 4$ then the last quadratic equation has two solutions.

\begin{thm}\label{th1}
Let $\tau=J\beta+\frac{1}{J\beta}$. Then for the model (\ref{H}) on the the Cayley tree of order two the following assertions hold:
\begin{enumerate}
\item If $\tau \leq 4$, then there is precisely one GGM associated to a 2-periodic boundary law.
\item If  $\tau > 4$, then there are precisely three such GGMs.
\end{enumerate}
\end{thm}

It is important to consider Gradient Gibbs measures associated with a 4-periodic boundary law for the model (\ref{H}). The following theorem gives us a full description of Gradient Gibbs measures associated to a 4-periodic boundary law.

\begin{thm}\label{th2}
Let $\tau=J\beta+\frac{1}{J\beta}$, $\tau_{cr}^{(1)}\approx 5.73$, and $\tau_{cr}^{(2)}\approx 6.261$. Then for Gradient Gibbs measures associated with a 4-periodic boundary law for the model (\ref{H}) on the Cayley tree of order two the following statements hold:
\begin{enumerate}
\item If $\tau\leq\tau_{cr}^{(1)}$, then there is not any GGM.
\item If $\tau_{cr}^{(1)}<\tau<\tau_{cr}^{(2)}$, then there exists a unique GGM.
\item If  $\tau=\tau_{cr}^{(2)}$, then there are exactly two GGMs.
\item If $\tau_{cr}^{(2)}< \tau <8$, then there are exactly three GGMs.
\item If $\tau=8$, then there are exactly four GGMs.
\item If $\tau >8$, then there are exactly five such GGMs.
\end{enumerate}
\end{thm}

\begin{proof}
Now we consider the case $a\neq b$. Then the system of equations (\ref{uff}) can be written as

\begin{equation} \label{se3}
 \begin{cases}
  a^{3}+ab^{2}+(\tau+2)a=\tau a^{2}+2b^{2}+2;\\[2mm]
  b^{3}+a^{2}b+(\tau+2)b=\tau b^{2}+2a^{2}+2.
   \end{cases}
\end{equation}
Now we subtract the second equation of (\ref{se3}) from the first one and get
$$a^{3}-b^{3}-ab(a-b)+(\tau+2)(a-b)=(\tau-2)(a^{2}-b^{2}).$$
Since $a\neq b$, both sides can be divided by $a-b$ and one gets
\begin{equation} \label{se4}
a^{2}+b^{2}+\tau+2=(\tau-2)(a+b).\end{equation}
By adding the second and first equations of (\ref{se3}), we have
\begin{equation}\label{se5}
(\tau-2)ab=2(\tau-2)(a+b)-2(\tau+1)
\end{equation}
Let $a+b=x$ and $ab=y$. By using (\ref{se4}) and (\ref{se5}) one gets a new system of equations with respect to $x$ and $y$ that is equivalent to (\ref{se3}):
\begin{equation} \label{se6}
 \begin{cases}
  x^{2}-2y+\tau+2=(\tau-2)x\\
  (\tau-2)y=2(\tau-2)x-2(\tau+1)
   \end{cases}
\end{equation}
In order to find the number of solutions of the last system we can consider the following quadratic equation with respect to $x$:
\begin{equation}\label{se7}
(\tau-2)x^{2}-(\tau^{2}-4)x+\tau^{2}+4\tau=0.
\end{equation}

It is easy to check that $$x_{1}=\frac{(\tau-2)(\tau+2)-\sqrt{(\tau-2)(\tau^{3}-2\tau^{2}-20\tau-8)}}{2(\tau-2)}$$ and $$x_{2}=\frac{(\tau-2)(\tau+2)+\sqrt{(\tau-2)(\tau^{3}-2\tau^{2}-20\tau-8)}}{2(\tau-2)}$$
are solutions to the equation (\ref{se7}).
Put $P(\tau)=\tau^{3}-2\tau^{2}-20\tau-8$. From $\tau>0$ it is sufficient to find only positive roots of $P(\tau)$. By Descartes' theorem (e.g. \cite{12}) $P(\tau)$ has at most one positive root. On the other hand, $P(5)<0$ and $P(6)>0$ i.e., by Intermediate Value Theorem $P(\tau)$ has at least one root in the segment $[5,6]$. Hence $P(\tau)$ has exactly one positive root. Let $\tau_{cr} (\tau_{cr}^{(1)}\approx 5.73)$ be the positive root of the polynomial. Consequently, we can conclude that if $2<\tau< \tau_{cr}^{(1)}$ then the system of equations (\ref{se6}) has not any positive solution. Let $\tau=\tau^{(1)}_{cr}$, then the system (\ref{se6}) has exactly one positive root.  For the case $\tau>\tau_{cr}^{(1)}$, then (\ref{se6}) has exactly two positive roots if we can show $x_1>0$. Namely, after short calculations, if $\tau>\tau_{cr}^{(1)}$ then we can show the inequality
$$\frac{(\tau-2)(\tau+2)-\sqrt{(\tau-2)(\tau^{3}-2\tau^{2}-20\tau-8)}}{2(\tau-2)}>0$$
is equivalent to the inequality $\tau^2+4\tau>0$.

For the case $\tau>\tau_{cr}^{(1)}$, from $a+b=x_i$ and $ab=y_i$ ($i\in\{1,2\}$) after short calculations, we have two quadratic equations respectively to $x_1$ and $x_2$:
$$2(\tau-2)a^{2}-(\tau^{2}-4-\sqrt{(\tau-2)(\tau^{3}-2\tau^{2}-20\tau-8)})a+ \ \ \ \  \ \ \  \ \  \ \ \ \  \ \ \ \ \ \ \  \ \ \  \ \ \  \  \ \ \ \  \ \ \ \ \ \ \  \ \ \  \ \ \ $$
$$\ \ \ \  \ \ \  \ \  \ \ \ \  \ \ \ \ \ \ \  \ \ \  \ \ \  \  \ \ \ \  \ \ \ \ \ \ \  \ \ \  \ \ \ +2(\tau^{2}-4-\sqrt{(\tau-2)(\tau^{3}-2\tau^{2}-20\tau-8)}-2(\tau+1))=0$$
and
$$2(\tau-2)a^{2}-(\tau^{2}-4+\sqrt{(\tau-2)(\tau^{3}-2\tau^{2}-20\tau-8)})a+\ \ \ \  \ \ \  \ \  \ \ \ \  \ \ \ \ \ \ \  \ \ \  \ \ \  \  \ \ \ \  \ \ \ \ \ \ \  \ \ \  \ \ \ $$
$$\ \ \ \  \ \ \  \ \  \ \ \ \  \ \ \ \ \ \ \  \ \ \  \ \ \  \  \ \ \ \  \ \ \ \ \ \ \  \ \ \  \ \ \ +2(\tau^{2}-4+\sqrt{(\tau-2)(\tau^{3}-2\tau^{2}-20\tau-8)}-2(\tau+1))=0.$$
The discriminants are
$$D_{1,2}(\tau)=2(\tau-2)(\tau^3-8\tau^{2}+4\tau+40\pm(6-\tau)\sqrt{(\tau-2)(\tau^{3}-2\tau^{2}-20\tau-8)}).$$
Now we find positive zeroes of
$$\tau^3-8\tau^{2}+4\tau+40\pm (6-\tau)\sqrt{(\tau-2)(\tau^{3}-2\tau^{2}-20\tau-8)}=0.$$
A solution to the last equation is also the solution to the following equation:
$$\tau^4-16\tau^3+80\tau^2-160\tau+256=(\tau-8)\left(\tau^3-8 \tau^2+16\tau-32\right)=0.$$
The discriminant of the polynomial $Q(\tau):=\tau^3-8 \tau^2+16\tau-32$ is negative and that is why $Q(\tau)$ has a unique real root. Since
$Q(0)<0$ and $\lim_{\tau\rightarrow \infty}Q(\tau)=+\infty$ this root $\tau_{cr}^{(2)}$ is positive and $\tau_{cr}^{(2)}\approx 6.261$.
Consequently, 8 (resp. $\tau_{cr}^{(2)}$) is a unique real positive root of $D_1(\tau)$ ($D_2(\tau)$).

Finally, we consider the case $\tau=\tau^{(1)}_{cr}$. From above, it is sufficient to solve the following equation:
$$2(\tau-2)a^{2}-(\tau^{2}-4)a+2(\tau^{2}-2\tau-6))=0.$$
Its discriminant is
$$D(\tau)=2(\tau-2)(\tau^3-8\tau^{2}+4\tau+40).$$
It's easy to check $D(\tau^{(1)}_{cr})<0$, thus there no positive solution to (\ref{se6}).
\end{proof}

\section*{Acknowledgements}
The work supported by the fundamental project (number: F-FA-2021-425)  of The Ministry of Innovative Development of the Republic of Uzbekistan. I thank Professor U.A.Rozikov for useful discussions and suggestions
which have improved the paper.

\section*{Statements and Declarations}

{\bf	Conflict of interest statement:}
The author states that there is no conflict of interest.

\section*{Data availability statements}
The datasets generated during and/or analyzed during the current study are available from the corresponding author upon reasonable request.

\end{document}